\documentclass{amsart}

\usepackage{amsthm,amsxtra}
\usepackage{amssymb}
\usepackage{amsmath}
\usepackage{amsfonts}

\usepackage{graphicx}
\usepackage{xcolor}

\allowdisplaybreaks

\theoremstyle{plain}
\newtheorem{theorem}{Theorem}[section]
\newtheorem{proposition}[theorem]{Proposition}

\newtheorem{corollary}[theorem]{Corollary}

\theoremstyle{definition}
\newtheorem{definition}[theorem]{Definition}
\newtheorem*{example} {Example}

\theoremstyle{remark}
\newtheorem{remark}[theorem]{Remark}

\def\to{\rightarrow}

\def\Aut{\operatorname{Aut}}
\def\Hom{\operatorname{Hom}}

\def\Sym{\operatorname{Sym}}

\def\dim{\operatorname{dim}}

\def\Gr{\operatorname{Gr}}
\def\SL{\operatorname{SL}}
\def\PSL{\operatorname{PSL}}
\def\GL{\operatorname{GL}}
\def\PGL{\operatorname{PGL}}
\def\SO{\operatorname{SO}}

\def\Spin{\operatorname{Spin}}

\def\Sp{\operatorname{Sp}}

\def\LGr{\operatorname{LGr}}

\begin{document}

\title[Deformation rigidity of smooth projective symmetric varieties]
{On the deformation rigidity of \\smooth projective symmetric varieties \\with Picard number one \\ 
\it
Sur la d\'eformation rigidit\'e de \\ vari\'et\'es projectives lisses sym\'etriques \\ de nombre de Picard un}

\author{Shin-Young Kim and Kyeong-Dong Park}

\address{Shin-Young Kim \\ Institut Fourier \\ Grenoble 38058, France 
\newline {\it Current Address}: Center for Geometry and Physics \\ Institute for Basic Science (IBS) \\ Pohang 37673, Korea}
\email{Shinyoung.Kim@univ-grenoble-alpes.fr, shinyoungkim@ibs.re.kr}

\address{Kyeong-Dong Park \\ Center for Geometry and Physics \\ Institute for Basic Science (IBS) \\ Pohang 37673, Korea}
\email{kdpark@ibs.re.kr} 

\thanks{The first author was supported by Basic Science Research Program through the National Research Foundation of Korea (NRF) funded by the Ministry of Education (2018R1A6A3A03012791), and by Institute for Basic Science in Korea (IBS-R003-D1). 
The second author was supported by IBS-R003-Y1 and IBS-R003-D1, Institute for Basic Science in Korea.}



\begin{abstract}
Symmetric varieties are normal equivariant open embeddings of symmetric homogeneous spaces and they are interesting examples of spherical varieties. 
The principal goal of this article is to study the rigidity under K\"{a}hler deformations of smooth projective symmetric varieties with Picard number one. \\
\\
R\'esum\'e. 
Les vari\'et\'es sym\'etriques sont les plongements ouverts normaux \'equivariants des espaces  homog\`enes sym\'etriques et ce sont des exemples int\'eressants de vari\'et\'es 
sph\'eriques. 
L'objectif principal de cet article est d'\'etudier la rigidit\'e sous les d\'eformations k\"{a}hleriennes des vari\'et\'es projectives lisses sym\'etriques de nombre de Picard un.
\end{abstract}

\maketitle

\section{Introduction}

For a connected semisimple algebraic group $G$ over $\mathbb C$ and an involution $\theta$ of $G$, 
the homogeneous space $G/H$ is called a \emph{symmetric homogeneous space}, 
where $H$ is a closed subgroup of $G$ such that $G^{\theta} \subset H \subset N_G(G^{\theta})$ 
(see Section 2.1 for details).
A normal $G$-variety $X$ together with an equivariant open embedding $G/H \hookrightarrow X$ 
of a symmetric homogeneous space $G/H$ is called a \emph{symmetric variety}. 
Our interest in this paper is the rigidity property under K\"{a}hler deformation 
of smooth projective symmetric varieties with Picard number one. 

From the Kodaira-Spencer deformation theory (cf. \cite{Kodaira}), 
the vanishing of the first cohomology group $H^1(G/P, T_{G/P})$ of a rational homogeneous manifold $G/P$ 
for a parabolic subgroup $P\subset G$ implies 
the local deformation rigidity of $G/P$. 
The \emph{global deformation rigidity} of a rational homogeneous manifold $G/P$ with Picard number one was studied by Hwang and Mok
in \cite{HwM98}, \cite{HM02}, \cite{HwM04b}, \cite{HM05}, \cite{HL2019}:  
a rational homogeneous manifold 
with Picard number one, 
different from the orthogonal isotropic Grassmannian $\Gr_q(2, 7)$, is globally rigid. 
This result can be generalized to some kinds of quasi-homogeneous varieties, 
for example, odd Lagrangian Grassmannians \cite{Park} and odd symplectic Grassmannians \cite{HL2019} among smooth projective horospherical varieties with Picard number one.
It is then natural to ask the same questions about smooth projective symmetric varieties.
Recently, the local deformation rigidity has been proven for smooth projective symmetric varieties with Picard number one, 
whose restricted root system is of type $A_2$ in \cite[Proposition 8.4]{FuHw2018-2} or \cite[Theorem 1.1]{BFM}. 
We obtain the global deformation rigidity of two smooth projective symmetric varieties of type $A_2$ under the assumption that the central fibers, of their deformation families, 
are not equivariant compactifications of the vector group $\mathbb C^{n}$, where $n$ is the dimension of the symmetric varieties.

\begin{theorem}\label{Main theorem} 
\label{Deformation rigidity of symmetric varieties}
Let $\pi \colon \mathcal X \rightarrow \Delta$ be a smooth projective morphism from a complex manifold $\mathcal X$
to the unit disc $\Delta\subset\mathbb C$. 
Denote by $S$ the smooth equivariant completion with Picard number one of the symmetric homogeneous space 
$\SL(6,\mathbb C)/ \Sp(6, \mathbb C)$ or $E_6/ F_4$. 
Suppose for any $t\in \Delta \backslash \{0\}$,
the fiber $\mathcal X_t=\pi^{-1}(t)$ is biholomorphic to the smooth projective symmetric variety $S$.
Then the central fiber $\mathcal X_0$ is biholomorphic to either $S$ or 
an equivariant compactification of the vector group $\mathbb C^n$, $n = \dim S$.
\end{theorem}

According to Theorem 2 of \cite{Ruzzi2010}, when smooth projective symmetric varieties with Picard number one have a restricted root system of type $G_2$, 
they are the smooth equivariant completions of either 
$G_2/(\SL(2, \mathbb C) \times \SL(2, \mathbb C))$
or $(G_2 \times G_2)/G_2$.
Recently, the smooth equivariant completion of 
$G_2/(\SL(2, \mathbb C) \times \SL(2, \mathbb C))$, which is called the \emph{Cayley Grassmannian}, 
has been studied by Manivel \cite{M}. 
Combining geometric descriptions of the Cayley Grassmannian in \cite{M} with the normal exact sequence leads to the local deformation rigidity.

\begin{theorem}
\label{Deformation rigidity of Cayley Grassmannian}
The smooth equivariant completion $S$ with Picard number one of the symmetric homogeneous space $G_2/(\SL(2, \mathbb C) \times \SL(2, \mathbb C))$ is locally rigid. 
\end{theorem}

In Section 2, we will review the classification results of smooth projective symmetric varieties with Picard number one and some general results about the variety of minimal rational tangents (VMRT). 
Moreover, we will prove the deformation rigidity of VMRT as a projective manifold under the assumption in Theorem \ref{Main theorem}. 
In Section 3, we relate the automorphism group of a projective variety with the prolongations of the Lie algebra of infinitesimal automorphisms of the cone structure given by its VMRT. 
By considering the smooth projective symmetric varieties with Picard number one of type $A_2$ and the affine cones of their VMRTs, we can prove Theorem \ref{Main theorem}. 
In Section 4, we prove Theorem \ref{Deformation rigidity of Cayley Grassmannian} using the Koszul complex associated to the Cayley Grassmannian and the Borel-Weil-Bott theorem. 

\vskip 1em

\noindent
\textbf{Acknowledgements}.
The authors would like to thank Michel Brion, Laurent Manivel, Jaehyun Hong, Nicolas Perrin, Baohua Fu and Qifeng Li for discussions on this topic and useful comments. 
The first author is especially appreciate to Michel Brion and Institut Fourier for her one year visiting during 2018/19. 
The second author is also grateful to Institut Fourier in Grenoble for his visiting, where his part of the present work was mostly completed in October 2018. 
We are grateful to the referee for his/her helpful comments that improves the presentation of this paper.

\section{Symmetric varieties and VMRT}


\subsection{Smooth projective symmetric varieties with Picard number one}

Let $G$ be a connected semisimple algebraic group over $\mathbb C$
and $\theta$ be an \emph{involution} of $G$, i.e., 
a nontrivial automorphism $\theta \colon G \to G$ such that $\theta^2 = id$. 

\begin{definition} 
Let $G^{\theta}=\{ g \in G : \theta(g)=g \}$. 
\begin{enumerate}
\item 
When $H$ is a closed subgroup of $G$ such that $G^{\theta} \subset H \subset N_G(G^{\theta})$, 
we say that the homogeneous space $G/H$ is a \emph{symmetric (homogeneous) space}. 
Here, $N_G(G^{\theta})$ means the normalizer of $G^{\theta}$ in $G$. 

\item
A normal $G$-variety $X$ together with an equivariant open embedding $G/H \hookrightarrow X$ 
of a symmetric space $G/H$ is called a \emph{symmetric variety}.
\end{enumerate}
\end{definition}

\begin{example} 
(1) For $G=\SL(n, \mathbb C) \times \SL(n, \mathbb C)$ and the involution $\theta(x, y)=(y,x)$, 
$G^{\theta} 
= \{ (x, x) \in \SL(n, \mathbb C) \times \SL(n, \mathbb C) \} \cong \SL(n, \mathbb C)$. 
In particular, if $n=2$ and $H=G^{\theta}$, 
then the symmetric space $G/H \cong \SL(2, \mathbb C)$ is a closed subvariety of $Mat_{2\times 2}(\mathbb C) \cong \mathbb C^4$. 
Let's consider an equivariant open embedding of $G/H$: 
\begin{eqnarray*}
G/H 
&\hookrightarrow & X:=\{ [x : t] : \det(x) = t^2 \} \subset \mathbb P(Mat_{2\times 2}(\mathbb C) \oplus \mathbb C) \\
x &\mapsto & [x:1].
\end{eqnarray*}
Thus, the symmetric variety $X$ is the 3-dimensional hyperquadric $\mathbb Q^3 \subset \mathbb P^4$.

(2) For $G=\SL(3, \mathbb C)$ and the involution $\theta(g)=(g^t)^{-1}$, we get $G^{\theta} = \SO(3, \mathbb C)$. 
The irreducible representation $V_{\SL(3, \mathbb C)}(2\varpi_1) = \Sym^2 \mathbb C^3$ 
is decomposed into $\Sym^2 \mathbb C^3 \cong V_{\SO(3, \mathbb C)}(4\varpi_1) \oplus V_{\SO(3, \mathbb C)}(0) = \mathbb C^5 \oplus \mathbb C$ as $\SO(3, \mathbb C)$-modules. 
From this result, we have an equivariant open embedding 
$\SL(3, \mathbb C)/N_G(\SO(3, \mathbb C)) \hookrightarrow \mathbb P(\Sym^2 \mathbb C^3) \cong \mathbb P^5 = X.$
\end{example}

Vust \cite[Theorem 1 in Section 1.3]{Vust} proved that a symmetric space $G/H$ is \emph{spherical}, i.e., 
it has an open orbit under the action of a Borel subgroup of $G$.
By using the Luna-Vust theory on spherical varieties, 
Ruzzi \cite{Ruzzi2011} classified the smooth projective symmetric varieties with Picard number one using colored fans.

\begin{theorem}[Theorem 1 of \cite{Ruzzi2010}]
Let $X$ be a smooth equivariant completion of a symmetric space $G/H$ with Picard number one. 
Then $X$ is nonhomogeneous if and only if 
\begin{enumerate}
\item
the restricted root system $\{\alpha - \theta(\alpha) : \alpha \in R_G\} \backslash \{ 0 \}$ has type either $A_2$ or $G_2$, 
where $R_G$ denotes the root system of $G$, and 
\item
$H=G^{\theta}$ (the closed subgroup of invariants of $\theta$).
\end{enumerate}
Given a symmetric space $G/H$, there is at most one embedding of $G/H$ with these properties. 
Furthermore, all these varieties are projective and Fano. 
\end{theorem}

Moreover, Ruzzi gave a geometric description of smooth projective symmetric varieties with Picard number one 
whose restricted root system is of type $A_2$ (Theorem 3 of \cite{Ruzzi2010}): 
these $A_2$-type symmetric varieties are smooth equivariant completions of symmetric homogeneous spaces 
$\SL(3, \mathbb C)/\SO(3, \mathbb C)$, 
$(\SL(3, \mathbb C) \times \SL(3, \mathbb C))/\SL(3, \mathbb C)$, 
$\SL(6, \mathbb C)/\Sp(6, \mathbb C)$, $E_6/F_4$, 
and are isomorphic to a general hyperplane section of rational homogeneous manifolds 
which are in the third row of the \emph{geometric Freudenthal-Tits magic square} (see \cite{Freudenthal}, \cite{LM}, and Section 3.5 of \cite{LM04}), respectively. 

\begin{center}
\begin{tabular}{ c | c c c c }
\hline
   & $\mathbb R$ & $\mathbb C$ & $\mathbb H$ & $\mathbb O$  \\
 \hline
  $\mathbb R$ & $v_4(\mathbb P^1)$ & $\mathbb P (T_{\mathbb P^2})$ & $\Gr_{\omega}(2, 6)$ & $\mathbb{OP}^2_0$  \\
  $\mathbb C$ & $v_2(\mathbb P^2)$ & $\mathbb P^2 \times \mathbb P^2$ & $\Gr(2, 6)$ & $\mathbb{OP}^2$ \\
  $\mathbb H$ & $\LGr(3, 6)$ & $\Gr(3, 6)$ & $\mathbb S_6$ & $E_7/P_7$ \\
  $\mathbb O$ & $F_4^{ad}$ & $E_6^{ad}$ & $E_7^{ad}$ & $E_8^{ad}$ \\
\hline
\end{tabular}
\end{center}

The fourth row ($\mathbb O$-row) of the square 
consists of the adjoint varieties for the exceptional simple Lie groups except $G_2$.
Taking the varieties of lines through a point, 
one obtains the third row which are \emph{Legendre varieties}. 
The second row is deduced from the third row by the same process, which consists of \emph{Severi varieties}. 
Then by taking general hyperplane sections we get the first row of the square.

\subsection{Variety of minimal rational tangents}

In 1990's Hwang and Mok introduced the notion of the \emph{variety of minimal rational tangents} on uniruled projective manifolds 
(see \cite{HM99} and \cite{Hw00}).
For the study of Fano manifolds, more generally uniruled manifolds,
a basic tool is the deformation 
of rational curves.
The study of the deformation of minimal rational curves
leads to their associated variety of minimal rational tangents
which is defined as the subvariety of the projectivized tangent bundle $\mathbb P (T_X)$
consisting of tangent directions of minimal rational curves
immersed in an uniruled projective manifold $X$.

Let $X$ be a projective manifold of dimension $n$.
By a \emph{parameterized rational curve} we mean a nonconstant
holomorphic map $f \colon \mathbb P^1 \to X$ from the projective line $\mathbb P^1$ into $X$.
We say that a (parameterized) rational curve $f \colon \mathbb P^1 \to X$ is \emph{free} 
if the pullback $f^* T_X$ of the tangent bundle is nonnegative 
in the sense that $f^* T_X$ splits into a direct sum 
$\mathcal O(a_1) \oplus \cdots \oplus \mathcal O(a_n)$ of
line bundles of degree $a_i \geq 0$ for all $i=1, \cdots, n$.
For a polarized uniruled projective manifold $(X, L)$ with an ample line bundle $L$,
a {\it minimal rational curve} on $X$ is a free rational curve of minimal degree among all free rational curves on $X$.

Let $\mathcal J$ be a connected component of the space of minimal rational curves
and let $\mathcal K:=\mathcal J /\Aut(\mathbb P^1)$ be the quotient space of unparameterized minimal rational curves.
We call $\mathcal K$ a {\it minimal rational component}.
For a point $x \in X$
consider the subvariety $\mathcal K_x$ of $\mathcal K$ consisting 
of minimal rational curves belonging to $\mathcal K$ marked at $x$.
Define the \emph{(rational) tangent  map} $\tau_x \colon \mathcal K_x \dashrightarrow \mathbb P(T_xX)$ by
$\tau_x([f(\mathbb P^1)])=[df(T_o\mathbb P^1)]$
sending a member of $\mathcal K_x$ smooth at $x$ to its tangent direction at $x$, 
where $f \colon \mathbb P^1 \rightarrow X$ is a minimal rational curve with $f(o)=x$.
For a general point $x \in X$, 
by Theorem 3.4 of \cite{Ke}, this tangent map induces a morphism 
$\tau_x \colon \mathcal K_x \rightarrow \mathbb P(T_xX)$
which is finite over its image.

\begin{definition}
Let $X$ be a polarized uniruled projective manifold with a minimal rational component $\mathcal K$.
For a general point $x \in X$, the image
$\mathcal C_x(X) :=\tau_x(\mathcal K_x)\subset \mathbb P(T_xX)$
is called the {\it variety of minimal rational tangents}
(to be abbreviated as VMRT) of $X$ at $x$. 
The union of $\mathcal C_x$ over general points $x\in X$ gives
the fibered space $\mathcal C \subset \mathbb P(T_X) \to X$
of varieties of minimal rational tangents associated to $\mathcal K$.
\end{definition}

From now on, $S$ denotes the smooth equivariant completion with Picard number one of 
the symmetric homogeneous space $\SL(6,\mathbb C)/ \Sp(6, \mathbb C)$ or $E_6/ F_4$, respectively. 

\begin{proposition} 
\label{VMRT of S}
For a general point $s \in S$, the VMRT $\mathcal C_s(S)$ of $S$ is projectively equivalent to 
$
\Gr_{\omega}(2, 6) \cong C_3/P_2 \subset \mathbb P^{13}$ or 
$\mathbb{OP}^2_0 \cong F_4/P_4 \subset \mathbb P^{25}$, respectively. 
Here, $P_k \subset G$ means the $k$-th maximal parabolic subgroup of $G$ following the Bourbaki ordering. 
\end{proposition}

\begin{proof}
For a nonsingular projective variety $X$ covered by lines and a general hyperplane section $X \cap H$, 
if $\mathcal C_x \subset \mathbb P (T_x X)$ is the VMRT of $X$ at a general point $x \in X \cap H$ and $\dim \mathcal C_x$ is positive, 
then the VMRT associated to a family of lines covering $X \cap H$ is 
$\mathcal C_x \cap \mathbb P (T_x H) \subset \mathbb P T_x(X \cap H)$
from Lemma 3.3 of \cite{FuHw}.

From Theorem 3 of \cite{Ruzzi2010}, 
the smooth equivariant completion $S$ with Picard number one of the symmetric space $\SL(6,\mathbb C)/ \Sp(6, \mathbb C)$ 
is isomorphic to a general hyperplane section of the 15-dimensional spinor variety $\mathbb S_6$. 
It is known that the VMRT of a rational homogeneous manifold $G/P$ associated to a \emph{long} simple root $\alpha_i$
is isomorphic to the \emph{highest weight variety} defined by the isotropy representation of a Levi factor of $P$ 
from Proposition 1 of \cite{HM02}. 
Note that, in this case, the VMRT of $G/P$ at the base point is the homogeneous manifold associated to 
the marked Dynkin diagram having markings corresponding to the simple roots which are adjacent to $\alpha_i$
in the Dynkin diagram of the semisimple part of $P$.
Thus, the VMRT of $\mathbb S_6$ is isomorphic to the Grassmannian $\Gr(2, 6)$.
Since we have the Pl\"{u}cker embedding of the Grassmannian 
$\Gr(2, 6) \hookrightarrow \mathbb P(\wedge^2 \mathbb C^6) \cong \mathbb P^{14}$ 
and a general hyperplane $H$ in $\mathbb P(\wedge^2 \mathbb C^6)$ is given by 
the kernel of a nondegenerate skew-symmetric 2-form $\omega \in (\wedge^2 \mathbb C^6)^*$, 
the VMRT $\mathcal C_s(S)$ is projectively equivalent to 
the symplectic isotropic Grassmannian $\Gr_{\omega}(2, 6) \subset \mathbb P^{13}$. 

Similarly, the smooth equivariant completion with Picard number one of $E_6/ F_4$ 
is isomorphic to a general hyperplane section of the 27-dimensional Hermitian symmetric space $E_7/P_7$ of compact type by Theorem 3 of \cite{Ruzzi2010}. 
Because the VMRT of 
$E_7/P_7$ is isomorphic to $E_6/P_6 \cong \mathbb{OP}^2$, 
the result follows from the standard facts on the geometric Freudenthal-Tits square summarized in Section 2.1. 
\end{proof}

\begin{corollary} \label{VMRT}
Let $\pi \colon \mathcal X \rightarrow \Delta$ be a smooth projective morphism from a complex manifold $\mathcal X$
to the unit disc $\Delta\subset\mathbb C$. 
Suppose for any $t\in \Delta \backslash \{0\}$,
the fiber $\mathcal X_t=\pi^{-1}(t)$ is biholomorphic to the smooth projective symmetric variety $S$.
Then the VMRT of the central fiber $\mathcal X_0$ at a general point $x$ is projectively equivalent to 
$\Gr_{\omega}(2, 6) \subset \mathbb P^{13}$ or 
$\mathbb{OP}^2_0 \subset \mathbb P^{25}$, respectively. 
\end{corollary}

\begin{proof}
Choose a section $\sigma \colon \Delta \rightarrow \mathcal X$ of $\pi$ such that $\sigma(0)=x$ and $\sigma(t)$ passes through a general point in $S$ for $t \neq 0$. 
Let $\mathcal K_{\sigma(t)}$ be the normalized Chow space of minimal rational curves passing through $\sigma(t)$ in $\mathcal X_t$. 
Then the canonical map $\mu \colon \mathcal K_{\sigma} \rightarrow \Delta$ given by the family $\{ \mathcal K_{\sigma(t)} \}$ 
is smooth and projective by the same proof as that of Proposition 4 of \cite{HwM98}. 
The main theorem of \cite {HM05} says that 
$\mathcal K_{\sigma(t)}$ is isomorphic to $\mathcal K_{s}(S)$ for all $t \in \Delta$ 
because 
$\mathcal K_{s}(S) \cong \mathcal C_{s}(S)$ is a rational homogeneous manifold with Picard number one by Proposition \ref{VMRT of S}. 
Thus it suffices to show that the image of the tangent map for the central fiber is 
nondegenerate in $\mathbb P(T_x \mathcal X_0)$, 
that is, the image is not contained in any hyperplane of the projective space $\mathbb P(T_x \mathcal X_0)$. 

Since $\dim \Gr_{\omega}(2, 6) = 7 > \frac{1}{2}\dim (\SL(6,\mathbb C)/ \Sp(6, \mathbb C) ) - 1 = 6$ 
and $\dim \mathbb{OP}^2_0 = 15 > 
\frac{1}{2}\dim (E_6/ F_4) - 1 = 12$, 
the distribution spanned by the VMRTs is integrable by Zak's theorem on tangencies 
(\cite{Zak} and Proposition 1.3.2 of 
\cite{HM99}).
Since the second Betti number $b_2(\mathcal X_0)=1$, 
the VMRT $\mathcal C_x (\mathcal X_0)$ at a general point $x$ is 
nondegenerate in $\mathbb P(T_x \mathcal X_0)$ by Proposition 13 of \cite{HwM98}. 
\end{proof}

\section{Prolongations 
of cone structure defined by VMRT and proof of Theorem \ref{Main theorem}}

\subsection{Prolongations of a linear Lie algebra}
Let $M$ be a differentiable manifold. 
Fix a vector space $V$ with $\dim V = \dim M$.
A \emph{frame} at $x \in M$ is a linear isomorphism $\sigma \colon V \rightarrow T_xM$. 
A \emph{frame bundle} $\mathcal F(M)$ on $M$ is the set of all frames $ \mathcal F_x(M) := \mbox{Isom}(V, T_xM) $  at every point $x \in M$. 
Then $\mathcal F(M)$ is a principal $\GL(V)$-bundle on $M$.
For a closed Lie subgroup $G\subset \GL(V)$, 
a \emph{(geometric) $G$-structure} on $M$ is defined as 
a $G$-subbundle $\mathcal G \subset \mathcal F(M)$ of the frame bundle.
The subbundle $V \times G$ of the frame bundle $\mathcal F(V) = V \times \GL(V)$ is called the \emph{flat $G$-structure} on $V$ 
and the $G$-structure on $M$ is \emph{locally flat} if it is locally equivalent to the flat $G$-structure on $V$. 
The \emph{(algebraic) prolongations $\mathfrak g^{(k)}$} of 
a linear Lie algebra $\mathfrak g \subset \mathfrak{gl}(V)$ originate from the
higher order derivatives of the infinitesimal automorphisms of the flat $G$-structure on $V$.

\begin{definition}
Let $V$ be a complex vector space and
$\mathfrak g \subset \mathfrak{gl}(V)$ a linear Lie algebra.
For an integer $k\geq 0$, the space $\mathfrak g^{(k)}$,
called the \emph{k-th prolongation} of $\mathfrak g$, is
the vector space of symmetric multi-linear homomorphisms
$A \colon \Sym^{k+1}V \to V$ such that
for any fixed vectors $v_1, \cdots, v_k \in V$,
the endomorphism
$$v\in V \mapsto 
A(v_1,\cdots, v_k,v)\in V$$
belongs to the Lie algebra $\mathfrak g$.
That is, $\mathfrak g^{(k)}=\Hom(\Sym^{k+1} V, V)\cap\Hom(\Sym^k V, \mathfrak g)$.
\end{definition}

We are interested in the case where a Lie algebra $\mathfrak g$ is relevant to geometric contexts, 
in particular, the Lie algebra of infinitesimal linear automorphisms of the affine cone of an irreducible projective subvariety.

\begin{definition}
Let $Z \subset \mathbb P V$ be an irreducible 
projective variety.
The {\it projective automorphism group} of
$Z$ is $\Aut(Z)=\{ g\in \PGL(V) 
: g(Z)=Z \}$ and its Lie algebra is denoted by $\mathfrak{aut}(Z)$.
Denote by $\widehat Z \subset V$ the affine cone of $Z$ and
by $T_{\alpha}\widehat{Z}\subset V$
the affine tangent space at a smooth point $\alpha\in\widehat{Z}$.
The {\it Lie algebra of infinitesimal linear automorphisms} of $\widehat{Z}$
is defined by
\begin{align*}
\mathfrak{aut}(\widehat{Z})&
=\{A\in \mathfrak{gl}(V) 
: \mbox{exp}(t A)(\widehat{Z})\subset \widehat{Z}, \, t\in \mathbb C \}, 
\end{align*}
where $\mbox{exp}(t A)$ denotes the one-parameter group of linear automorphisms of $V$.
Its $k$-th prolongation
$\mathfrak{aut}(\widehat{Z})^{(k)}$ will be called the 
\emph{$k$-th prolongation of $Z \subset \mathbb P V$}.
\end{definition}

In \cite{HM05}, Hwang and Mok studied the prolongations $\mathfrak{aut}(\widehat{Z})^{(k)}$ 
of a projective variety $Z \subset \mathbb P V$
using the projective geometry of $Z$ 
and the deformation theory of rational curves on $Z$. 
In particular, the vanishing of the second prolongation $\mathfrak{aut}(\widehat{Z})^{(2)}$
for an irreducible smooth nondegenerate projective variety $Z$ embedded in the projective space $\mathbb P V$ was proven. 

\begin{proposition}[Theorem 1.1.2 of \cite{HM05}]
\label{second prolongation}
Let $Z \subset \mathbb P V$ be
an irreducible smooth nondegenerate projective variety.
If $Z \neq \mathbb P V$, then the second prolongation of $Z$ vanishes, that is,
$\mathfrak{aut}(\widehat{Z})^{(2)} = 0$.
\end{proposition}

From the definition of prolongations, it is immediate that
$\mathfrak g^{(k)}=0$ for some $k\geq 0$ implies $\mathfrak g^{(k+1)}=0$.
Thus if $Z \varsubsetneqq \mathbb P V$ is
an irreducible smooth nondegenerate projective variety,
then $\mathfrak{aut}(\widehat{Z})^{(k)} = 0$ for $k\geq 2$.

\subsection{Infinitesimal automorphisms of cone structures}

A \emph{cone structure} $\mathcal C$ on a complex manifold $M$ is
a closed analytic subvariety $\mathcal C \subset \mathbb P(T_M)$
such that the natural projection $\pi \colon \mathcal C \rightarrow M$
is proper, flat and surjective with connected fibers. We denote the fiber $\pi^{-1}(x)$ by $\mathcal C_x$ for a point $x \in M$. 
A germ of holomorphic vector field $v$ at $x \in M$
is said to \emph{preserve the cone structure}
if the local one-parameter family of biholomorphisms integrating $v$
lifts to local biholomorphisms of $\mathbb P(T_M)$ preserving $\mathcal C$.

\begin{definition}
Let $\mathcal C$ be a cone structure on a complex manifold $M$. 
The \emph{Lie algebra $\mathfrak{aut}(\mathcal C, x)$ of
infinitesimal automorphisms of the cone structure $\mathcal C$ at $x \in M$}
is the set of all germs of holomorphic vector fields
preserving the cone structure $\mathcal C$ at $x$.
\end{definition}

The Lie algebra $\mathfrak{aut}(\mathcal C, x)$ is naturally
filtered by the vanishing order of vector fields at $x$.
More precisely, for each integer $k\geq 0$, let $\mathfrak{aut}(\mathcal C, x)_k$ be
the subalgebra of $\mathfrak{aut}(\mathcal C, x)$
consisting of vector fields that vanish at $x$ to the order $\geq k+1$. 
The Lie bracket gives the structure of filtration
$$\mathfrak{aut}(\mathcal C, x)\supset
\mathfrak{aut}(\mathcal C, x)_0 \supset
\mathfrak{aut}(\mathcal C, x)_1 \supset
\mathfrak{aut}(\mathcal C, x)_2\supset \cdots.$$ 
Let $\xi$ be a germ of holomorphic vector field on $M$
vanishing to order $\geq k + 1$ at $x$.
Then its $(k+1)$-jet $J^{k+1}_x(\xi)$ 
defines an element of $\Sym^{k+1} (T_x^* M) \otimes T_x M$.
Because $J^{k+1}_x(\zeta)=0$ for a vector field $\zeta$
vanishing to order $\geq k + 2$ at $x$,
this defines the inclusion
$\mathfrak{aut}(\mathcal C, x)_k/\mathfrak{aut}(\mathcal C, x)_{k+1}
\subset \Hom(\Sym^{k+1} (T_x M), T_x M)$.
The following result follows from Proposition 1.2.1 of 
\cite{HM05}.

\begin{proposition}
\label{aut:inequality}
Let $\mathcal C \subset \mathbb P(T_M)$ be
a cone structure on a complex manifold $M$
and $x \in M$ a point.
For each $k \geq 0$, if the quotient space
$\mathfrak{aut}(\mathcal C, x)_k/\mathfrak{aut}(\mathcal C, x)_{k+1}$
is regarded as a subspace of $\Hom(\Sym^{k+1} (T_x M), T_x M)$,
then we have the inclusion
$$\mathfrak{aut}(\mathcal C, x)_k/\mathfrak{aut}(\mathcal C, x)_{k+1} \subset
\mathfrak{aut}(\widehat{\mathcal C_x})^{(k)}.$$
\end{proposition}

From Proposition \ref{aut:inequality}, we have the natural inequalities 
\begin{eqnarray*}
\dim \mathfrak{aut}(\mathcal C, x)_0
&\leq& \dim \mathfrak{aut}(\widehat{\mathcal C_x}) + \dim \mathfrak{aut}(\mathcal C, x)_1 \\
&\leq& \dim \mathfrak{aut}(\widehat{\mathcal C_x}) + \dim
\mathfrak{aut}(\widehat{\mathcal C_x})^{(1)} + \dim
\mathfrak{aut}(\mathcal C, x)_2\leq \cdots .
\end{eqnarray*}
Because the codimension of $\mathfrak{aut}(\mathcal C, x)_0$ in
$\mathfrak{aut}(\mathcal C, x)$ is at most $\dim M$,
we obtain the following direct consequence (see Proposition 5.10 of \cite{FuHw}).

\begin{corollary}
\label{aut} 
Let $\mathcal C \subset \mathbb P(T_M)$ be a cone structure on a complex manifold $M$
and $x \in M$.
If $\mathfrak{aut}(\widehat{\mathcal C_x})^{(k+1)} = 0$
for some $k \geq 0$,
then $$\dim \mathfrak{aut}(\mathcal C, x) \leq
\dim M +\dim \mathfrak{aut}(\widehat{\mathcal C_x}) +\dim
\mathfrak{aut}(\widehat{\mathcal C_x})^{(1)}+ \cdots + \dim
\mathfrak{aut}(\widehat{\mathcal C_x})^{(k)}.$$
\end{corollary}

\subsection{Cone structure defined by VMRT}

Let $Z \subset \mathbb PV$ be a (fixed) projective variety with $\dim V = \dim M$. 
A cone structure $\mathcal C \subset \mathbb P(T_M)$ is \emph{$Z$-isotrivial} 
if for a general point $x \in M$, 
the fiber $\mathcal C_x \subset \mathbb P(T_x M)$ is isomorphic to 
$Z \subset \mathbb PV$ as a projective variety, 
i.e., there exists a linear isomorphism $V \rightarrow T_xM$ sending $Z$ to $\mathcal C_x$.



For the affine cone $\widehat{Z} \subset V$ of $Z$, 
let $G = \Aut(\widehat{Z}) = \{ g \in \GL(V) : g(\widehat{Z})=\widehat{Z} \}$ be the automorphism group of $\widehat{Z} \subset V$.
A $Z$-isotrivial cone structure $\mathcal C$ on $M$ induces 
the $G$-structure $\mathcal G$ of cone type of which a fiber at general point $x$ is 
$$\mathcal G_x = \{ \sigma \in \mbox{Isom}(V, T_x M) \colon \sigma(\widehat{Z}) = \widehat{\mathcal C_x} \}.$$
An isotrivial cone structure $\mathcal C$ on $M$ is \emph{locally flat}
if its associated $G$-structure $\mathcal G$ is locally flat. 
We know that if $\mathcal C$ is a locally flat cone structure on $M$ with $\mathfrak{aut}(\widehat{\mathcal C_x}) = \mathfrak{g}$ 
and $k$ is a nonnegative integer such that $\mathfrak{aut}(\widehat{\mathcal C_x})^{(k+1)} = 0$, 
then $\mathfrak{aut}(\mathcal C, x)$ is isomorphic to the graded Lie algebra $V \oplus \mathfrak{g} \oplus \mathfrak{g}^{(1)} \oplus \cdots \oplus \mathfrak{g}^{(k)}$. 
Conversely, if the equality in Corollay \ref{aut} holds, then the cone structure $\mathcal C$ is locally flat by Corollary 5.13 of \cite{FuHw}.  

Now, we are ready to prove Theorem \ref{Main theorem} by considering the cone structure defined by VMRT which is $Z$-isotrivial.

\subsection{Proof of Theorem \ref{Main theorem}}
(i) If $S$ is the smooth equivariant completion with Picard number one of 
the symmetric homogeneous space $\SL(6,\mathbb C)/ \Sp(6, \mathbb C)$, 
then its automorphism group $\Aut(S)$ is generated by $\PSL(6, \mathbb C)$ 
and the involution $\theta$ with $\SL(6,\mathbb C)^{\theta}=\Sp(6, \mathbb C)$ by Proposition 3 of \cite{Ruzzi2010}.

From Corollary \ref{VMRT}, we can compute the Lie algebras of infinitesimal automorphisms of the affine cones of VMRTs: 
$\mathfrak{aut}(\widehat{\mathcal C_x(\mathcal X_0)}) = \mathfrak{aut}(\widehat{\mathcal C_s(S)})
\cong \mathfrak{sp}(6, \mathbb C) \oplus \mathbb C$. 
Since the variety $\mathcal C_s(S)$ of minimal rational tangents of $S$ 
is irreducible smooth nondegenerate and not linear,
$\mathfrak{aut}(\widehat{\mathcal C_s})^{(k)} = 0$ for $k\geq 2$ 
by Proposition \ref{second prolongation}.  
Furthermore, the classification of projective varieties with non-zero prolongation in \cite{FuHw2018}
implies that $\mathfrak{aut}(\widehat{\mathcal C_s})^{(k)}=0$ for all $k \geq 1$.
Thus, for the cone structure $\mathcal C$ on a fiber $\mathcal X_t$ given by its VMRT 
we have equalities
$$\dim \mathfrak{aut}(S) + 1 = \dim S+\dim \mathfrak{aut}(\widehat{\mathcal C_s}) = \dim \mathcal X_0 +\dim \mathfrak{aut}(\widehat{\mathcal C_x}).$$

Because the Lie algebra $\mathfrak{aut}(\mathcal X_0)$ is isomorphic to
the space $H^0(\mathcal X_0, T_{\mathcal X_0})$ of global sections of
the tangent bundle $T_{\mathcal X_0}$,
we know $h^0(\mathcal X_0, T_{\mathcal X_0})=\dim \mathfrak{aut}(\mathcal X_0)$. 
Since the action of $\Aut(\mathcal X_0)$ preserves the VMRT-structure $\mathcal C$ on $\mathcal X_0$, 
we have an inclusion $\mathfrak{aut}(\mathcal X_0) \subset \mathfrak{aut}(\mathcal C, x)$. 
Hence, from Corollary \ref{aut} we have inequalities
\begin{eqnarray*}
h^0(\mathcal X_0, T_{\mathcal X_0})
= \dim \mathfrak{aut}(\mathcal X_0)  & \leq & \dim \mathfrak{aut}(\mathcal C, x) \\
&\leq& \dim \mathcal X_0
+\dim \mathfrak{aut}(\widehat{\mathcal C_x})
\\
&=&\dim \mathfrak{aut}(S) +1 = h^0(S,T_S) +1.
\end{eqnarray*}
Now, recall the standard fact that the Euler-Poincar\'e characteristic 
of the holomorphic tangent bundle $T_X$ on a Fano manifold $X$ 
is given by $\chi(X,T_X)=h^0(X,T_X)-h^1(X,T_X).$ 
In fact, the Serre duality and Kodaira-Nakano vanishing theorem imply that 
$H^i(X, T_X)=H^{n-i}(X, T^*_X \otimes K_X)^* 
=0$ for $i\geq 2$.
Since the Euler-Poincar\'e characteristic is constant in a smooth family 
and we already know $h^1(S, T_S)=0$ by Proposition 8.4 of \cite{FuHw2018-2}, 
$h^1(\mathcal X_0, T_{\mathcal X_0}) = h^0(\mathcal X_0, T_{\mathcal X_0}) - h^0(S, T_S)\leq 1$.


Now, it suffices to consider two possible cases. 
Suppose that the above equality holds. 
Then we have $\dim \mathfrak{aut}(\mathcal C, x) = \dim \mathcal X_0 + \dim \mathfrak{aut}(\widehat{\mathcal C_x})$, 
which implies that the isotrivial cone structure $\mathcal C$ given by VMRT on the central fiber $\mathcal X_0$ should be locally flat by Corollary 5.13 of \cite{FuHw}. 
Thus $\mathcal X_0$ is an equivariant compactification of the vector group $\mathbb C^{14}$ from Theorem 1.2 of \cite{FuHw2018-2}. 
Next, if $h^1(\mathcal X_0,T_{\mathcal X_0})=0$, then the central fiber $\mathcal X_0$ is also biholomorphic to the general fiber $S$.  

\vskip 0.6em

(ii) If $S$ is the smooth equivariant completion with Picard number one of the symmetric homogeneous space $E_6/ F_4$,
then its automorphism group $\Aut(S)$ is generated by $E_6$ 
and the involution $\theta$ with $E_6^{\theta}=F_4$ by Proposition 3 of \cite{Ruzzi2010}. 
From Corollary \ref{VMRT}, 
$\mathfrak{aut}(\widehat{\mathcal C_x(\mathcal X_0)}) = \mathfrak{aut}(\widehat{\mathcal C_s(S)}) 
\cong \mathfrak{f_4} \oplus \mathbb C$. 
Because $\mathfrak{aut}(\widehat{\mathcal C_s})^{(k)}=0$ for all $k \geq 1$ by \cite{FuHw2018}, 
we also have the same equality as before:  
$$\dim \mathfrak{aut}(S) + 1 
= \dim S+\dim \mathfrak{aut}(\widehat{\mathcal C_s})
= \dim \mathcal X_0 +\dim \mathfrak{aut}(\widehat{\mathcal C_x}).$$
By Proposition 8.4 of \cite{FuHw2018-2}, a general hyperplane section of $E_7/P_7$ is locally rigid, 
so we see that $h^1(S, T_S)=0$.
Therefore, the same argument as (i) works immediately. 
\qed

\section{Local rigidity of smooth projective symmetric varieties of type $G_2$}

The smooth equivariant completion $S$ with Picard number one of the symmetric space $G_2/(\SL(2, \mathbb C) \times \SL(2, \mathbb C))$, 
called the \emph{Cayley Grassmannian}, has been studied by Manivel \cite{M}. 
The Cayley Grassmannian is a smooth projective variety parametrizing four-dimensional subalgebras of the complexified octonions $\mathbb O_{\mathbb C}$. 
Because all subalgebras contain the unit element, 
the Cayley Grassmannian is a closed subvariety of the Grassmannian $\Gr(3, 7)$ 
by considering only the imaginary parts. 
It can be also described as a subvariety of the Grassmannian $\Gr(4, 7)$ 
by mapping a subalgebra to its orthogonal complement contained in the imaginary part of $\mathbb O_{\mathbb C}$. 
From now, we will consider the Cayley Grassmannian as a subvariety of 
$\Gr(4, 7)$.

\begin{proposition}[Proposition 3.2 of \cite{M}] 
\label{Description of the Cayley Grassmannian} 
The Cayley Grassmannian $S$ 
is projectively equivalent to the zero locus of a general global section of the rank four vector bundle $\wedge^3 \mathcal U^*$ on the Grassmannian $\Gr(4, 7)$, 
where $\mathcal U$ denotes the universal subbundle of rank four on $\Gr(4, 7)$.
\end{proposition}

\begin{remark} 
The symmetric variety $S$ is a Fano eightfold of index 4 by the adjunction formula. 
Indeed, $K_S = K_{\Gr(4,7)} \otimes \mbox{det} (\wedge^3 \mathcal U^*) = \mathcal O(-7) \otimes \mathcal O(3) = \mathcal O(-4)$. 
This implies that the VMRT of $S$ at a general point is isomorphic to a surface embedded in $\mathbb P^7$. 
\end{remark}


From the Kodaira-Spencer deformation theory, 
it suffices to prove $H^1(S, T_S) = 0$ for Theorem \ref{Deformation rigidity of Cayley Grassmannian}.
Now, we recall Borel-Weil-Bott theorem to compute the cohomology groups of equivariant vector bundles on a rational homogeneous variety $G/P$. 

Let $G$ be a simply connected complex semisimple algebraic group and $P \subset G$ a parabolic subgroup. 
For an integral dominant weight $\omega$ with respect to $P$, 
we have an irreducible representation $V(\omega)$ of $P$ with the highest weight $\omega$, and 
denote by $\mathcal E_{\omega}$ the corresponding \emph{irreducible equivariant vector bundle} $G\times_P V(\omega)$ on $G/P$:
$$\mathcal E_{\omega}:= G \times_P V(\omega) = (G \times V(\omega))/P,$$
where the equivalence relation is given by $(g, v) \sim (gp, p^{-1} . v)$ for $p \in P$.

\begin{theorem}[Borel-Weil-Bott theorem \cite{Bo57}] 
Let $\rho$ denote the sum of fundamental weights of $G$.
\begin{itemize}
\item If a weight $\omega+\rho$ is singular, that is, it is orthogonal to some (positive) root of $G$, 
then all cohomology groups $H^i(G/P, \mathcal E_{\omega})$ vanish for all $i$. 
\item Otherwise, $\omega+\rho$ is regular, that is, it lies in the interior of some Weyl chamber, 
then $H^{\ell(w)}(G/P, \mathcal E_{\omega})=V_G(w(\omega+\rho)-\rho)^*$ and any other cohomology vanishes. 
Here, $w\in W$ is a unique element of the Weyl group of $G$ such that $w(\omega+\rho)$ is strictly dominant, 
and $\ell(w)$ means the length of $w \in W$, that is, the minimal integer $\ell(w)$ such that $w$ can be expressed as a product of $\ell(w)$ simple reflections.  
\end{itemize}
\end{theorem}

\noindent {\it Proof of Theorem \ref{Deformation rigidity of Cayley Grassmannian}.} 
Since $S$ is the zero locus of a general global section of $\wedge^3 \mathcal U^*$ on $\Gr(4, 7)$, 
we have the normal exact sequence on $S$ 
$$ 0 \to T_S \to T_{\Gr(4,7)}|_S \to \wedge^3 \mathcal U^* |_S \to 0$$
and the Koszul complex of the structure sheaf $\mathcal O_S$ 
$$0 \to \wedge^4 (\wedge^3 \mathcal U) \to \wedge^3 (\wedge^3 \mathcal U) \to \wedge^2 (\wedge^3 \mathcal U) \to  \wedge^3 \mathcal U \to \mathcal O_{\Gr(4,7)} \to \mathcal O_S \to 0.$$
Using the isomorphisms $\wedge^4 \mathcal U \cong \mathcal O(-1)$ and $\wedge^3 \mathcal U \cong \mathcal U^*(-1)$ on $\Gr(4, 7)$, 
we get an exact sequence 
$$0 \to \mathcal O_{\Gr(4,7)}(-3) \to \mathcal U(-2) \to \wedge^2 \mathcal U(-1) \to \wedge^3 \mathcal U \to \mathcal O_{\Gr(4,7)} \to \mathcal O_S \to 0.$$
Indeed, $\wedge^3 (\wedge^3 \mathcal U) \cong \wedge^3 (\mathcal U^*(-1)) = \wedge^3 \mathcal U^* \otimes \mathcal O(-3) 
\cong \mathcal U(1) \otimes \mathcal O(-3) = \mathcal U(-2)$.   
Taking the tensor product of the Koszul complex 
with $\wedge^3 \mathcal U^*$, we have 
$$0 \to \wedge^3 \mathcal U^*(-3) \to \mathcal U(-2) \otimes \wedge^3 \mathcal U^* 
\to \wedge^2 \mathcal U(-1) \otimes \wedge^3 \mathcal U^* \to \wedge^3 \mathcal U \otimes \wedge^3 \mathcal U^* 
\to \wedge^3 \mathcal U^* \to \wedge^3 \mathcal U^*|_S \to 0 .$$

Let $\omega_1, \cdots, \omega_6$ be the fundamental weights of $\SL(7, \mathbb C)$. 
Since $\wedge^3 \mathcal U^*(-3)$ is the irreducible equivariant vector bundle associated with the weight $\omega_3 - 3 \omega_4$  
and the weight $\omega_3 - 3 \omega_4+\rho$ is singular, as a straightforward application of the Borel-Weil-Bott theorem, 
we see that $H^{i}(\Gr(4,7), \wedge^3 \mathcal U^*(-3) )=0$ for all $i$.
Also, because 
$$\mathcal U(-2) \otimes \wedge^3 \mathcal U^* \cong \mathcal U(-2) \otimes \mathcal U(1) 
\cong (\wedge^2 \mathcal U \oplus \Sym^2 \mathcal U) \otimes \mathcal O(-1) = \mathcal E_{\omega_2 - 2\omega_4} \oplus \mathcal E_{2\omega_3 - 3\omega_4}$$ 
and both $\omega_2 - 2\omega_4+\rho$ and $2\omega_3 - 3\omega_4+\rho$ are singular weights, 
we have that $H^{i}(\Gr(4,7), \mathcal U(-2) \otimes \wedge^3 \mathcal U^* )=0$ for all $i$. 
From the Littlewood-Richardson rule (see Section 2.3 of \cite{W} for details), 
we can check that $\wedge^2 \mathcal U(-1) \otimes \wedge^3 \mathcal U^* \cong \mathcal E_{\omega_1 - \omega_4} \oplus \mathcal E_{\omega_2 + \omega_3 - 2\omega_4}$ 
which implies that $H^{i}(\Gr(4,7), \wedge^2 \mathcal U(-1) \otimes \wedge^3 \mathcal U^*)=0$ for all $i$. 
Since we know that $\wedge^3 \mathcal U \otimes \wedge^3 \mathcal U^* \cong \mathcal E_{\omega_1 + \omega_3 - \omega_4} \oplus \mathcal O$  by the Littlewood-Richardson rule, 
$H^i(\Gr(4,7), \wedge^3 \mathcal U \otimes \wedge^3 \mathcal U^* ) = 0$ for $i>0$ and 
$H^0(\Gr(4,7), \wedge^3 \mathcal U \otimes \wedge^3 \mathcal U^* ) = H^0(\Gr(4,7), \mathcal O)= \mathbb C$.
Again, the Borel-Weil theorem says that  $H^0(\Gr(4,7), \wedge^3 \mathcal U^* ) = \wedge^3 \mathbb C^7$. 
Therefore, we conclude $H^0(S, \wedge^3 \mathcal U^* |_S)=\wedge^3 \mathbb C^7 / \mathbb C$. 

Using the Littlewood-Richardson rule, 
we get the Koszul complex of the structure sheaf $\mathcal O_S$ tensored with the tangent bundle $T_{\Gr(4,7)}=\mathcal U \otimes \mathcal Q = \mathcal E_{\omega_1 + \omega_6}$:  
\begin{eqnarray*}
0  \ \to  \ \mathcal E_{\omega_1 -3\omega_4 + \omega_6} 
&\to& \mathcal E_{\omega_1 + \omega_3 - 3\omega_4 + \omega_6} \oplus \mathcal E_{- 2\omega_4 + \omega_6} \
\to \ \mathcal E_{\omega_1 + \omega_2 - 2\omega_4 + \omega_6} \oplus \mathcal E_{\omega_3 - 2\omega_4 + \omega_6} \\ 
&\to& \mathcal E_{2\omega_1 - \omega_4 + \omega_6} \oplus \mathcal E_{\omega_2 - \omega_4 + \omega_6} 
\ \to \ T_{\Gr(4,7)} \ \to  \ T_{\Gr(4,7)}|_S \ \to \ 0.
\end{eqnarray*}
Since all bundles except the last two terms are acyclic, using the Borel-Weil-Bott theorem again, 
we obtain $H^0(S, T_{\Gr(4,7)}|_S)= H^0(\Gr(4,7), T_{\Gr(4,7)}) = \mathfrak{sl}_7$ and $H^1(S, T_{\Gr(4,7)}|_S)= H^1(\Gr(4,7), T_{\Gr(4,7)})=0$. 


Then, from the normal exact sequence on $S$, we deduce an exact sequence 
$$0 \to H^0(S, T_S) \to \mathfrak{sl}_7 \to \wedge^3 \mathbb C^7 / \mathbb C \to H^1(S, T_S) \to 0.$$
Hence, $H^0(S, T_S) = \mathfrak{aut}(S) =\mathfrak g_2$ (Lemma 16 of \cite{Ruzzi2010}) implies $H^1(S, T_S)=0$ 
from which the local rigidity of $S$ follows by the Kodaira-Spencer deformation theory.
\qed

\begin{remark} 
By Theorem 2 of \cite{Ruzzi2010}, the smooth projective symmetric varieties with Picard number one whose restricted root system is of type $G_2$
are either the Cayley Grassmannian or the smooth equivariant completion of $(G_2 \times G_2)/G_2$.
Recently, Manivel also studied the latter, 
called the \emph{double Cayley Grassmannian}, 
and proved that it is locally rigid in \cite{M19}. 
The double Cayley Grassmannian is projectively equivalent to the zero locus of a general global section of the rank seven vector bundle $\mathcal U \otimes \mathcal L$ on the 21-dimensional spinor variety $\mathbb S_{14}=\Spin (14, \mathbb C)/P_7$, 
where $\mathcal U$ denotes the tautological bundle of rank seven on $\mathbb S_{14}$ and 
$\mathcal L$ is the very ample line bundle defining the minimal embedding $\mathbb S_{14} \hookrightarrow \mathbb P^{63}$.
Consequently, we conclude that all smooth projective symmetric varieties with Picard number one are locally rigid.  
On the other hand, the global deformation rigidity problem on smooth projective symmetric varieties of type $G_2$ remains open. 
\end{remark}


\begin{thebibliography}{99}

\bibitem {BFM} Chenyu Bai, Baohua Fu, and Laurent Manivel, {\it On Fano complete intersections in rational homogeneous varieties}, in press, Math. Z. (2019), https://doi.org/10.1007/s00209-019-02351-4. 

\bibitem {Bo57} Raul Bott, {\it Homogeneous vector bundles}, Ann. of Math. {\bf 66} (1957), no. 2, 203--248.

\bibitem {Freudenthal} Hans Freudenthal, {\it Lie groups in the foundations of geometry}, Adv. Math. {\bf 1} (1964), no. 2, 145--190.

\bibitem {FuHw} Baohua Fu and Jun-Muk Hwang, {\it Classification of non-degenerate projective varieties with non-zero prolongation and application to target rigidity}, Invent. Math. {\bf 189} (2012), no. 2, 457--513.

\bibitem {FuHw2018} Baohua Fu and Jun-Muk Hwang, {\it Special birational transformations of type (2,1)}, J. Algebraic Geom. {\bf 27} (2018), no. 1, 55--89.

\bibitem {FuHw2018-2} Baohua Fu and Jun-Muk Hwang, {\it Isotrivial VMRT-structures of complete intersection type}, Asian J. Math. {\bf 22} (2018), no. 2, 0333--0354.

\bibitem {Hw00} Jun-Muk Hwang, {\it Geometry of minimal rational curves on Fano manifolds}, School on Vanishing Theorems and Effective Results in Algebraic Geometry (Trieste, 2000), ICTP Lect. Notes, 6, Abdus Salam Int. Cent. Theoret. Phy., Trieste, 2001, 335--393.

\bibitem{HL2019} Jun-Muk Hwang and Qifeng Li, {\it Characterizing symplectic Grassmannians by varieties of minimal rational tangents}, preprint 2019, arXiv:1901.00357.

\bibitem {HwM98} Jun-Muk Hwang and Ngaiming Mok, {\it Rigidity of irreducible Hermitian symmetric spaces of the compact type under K\"ahler deformation}, Invent. Math. {\bf 131} (1998), no. 2, 393--418.

\bibitem {HM99} Jun-Muk Hwang and Ngaiming Mok, {\it Varieties of minimal rational tangents on uniruled projective  manifolds}, in Several Complex Variables (Berkeley, CA, 1995--1996), MSRI publications {\bf 37}, Cambridge University Press, 1999, 351--389.

\bibitem {HM02} Jun-Muk Hwang and Ngaiming Mok, {\it Deformation rigidity of the rational homogeneous space associated  to a long simple root}, Ann. Sci. \'Ecole Norm. Sup. {\bf 35} (2002), no. 2, 173--184.

\bibitem {HwM04b} Jun-Muk Hwang and Ngaiming Mok, {\it Deformation rigidity of the 20-dimensional $F\sb 4$-homogeneous space associated to a short root}, in Algebraic transformation groups and algebraic varieties, Encyclopaedia of Mathematical Sciences {\bf 132}, Springer, Berlin, 2004, 37--58.

\bibitem {HM05} Jun-Muk Hwang and Ngaiming Mok, {\it Prolongations of infinitesimal linear automorphisms of projective varieties and rigidity of rational homogeneous spaces of Picard number 1 under K\"{a}hler deformation}, Invent. Math. {\bf 160} (2005), no. 3, 591--645.

\bibitem {Ke} Stefan Kebekus, {\it Families of singular rational curves}, J. Algebraic Geom. {\bf 11} (2002), no. 2, 245--256.

\bibitem {Kodaira} Kunihiko Kodaira, {\it Complex manifolds and deformation of complex structures}, Grundlehren der mathematischen Wissenschaften {\bf 283}, Springer, Berlin-Heidelberg-New York, 1986.

\bibitem {LM} Joseph M. Landsberg and Laurent Manivel, {\it The projective geometry of Freudenthal's magic square}, J. Algebra {\bf 239} (2001), no. 2, 477--512.

\bibitem {LM04} Joseph M. Landsberg and Laurent Manivel, {\it Representation theory and projective geometry}, in Algebraic transformation groups and algebraic varieties, Encyclopaedia of Mathematical Sciences {\bf 132}, Springer, Berlin, 2004, 71--122.

\bibitem {M} Laurent Manivel, {\it The Cayley Grassmannian}, J. Algebra {\bf 503} (2018), 277--298.

\bibitem {M19} Laurent Manivel, {\it The double Cayley Grassmannian}, in preparation. 


\bibitem {Park} Kyeong-Dong Park, {\it Deformation rigidity of odd Lagrangian Grassmannians}, J. Korean Math. Soc. {\bf 53} (2016), no. 3, 489--501.

\bibitem {Ruzzi2010} Alessandro Ruzzi, {\it Geometrical description of smooth projective symmetric varieties with Picard number one}, Transform. Groups {\bf 15} (2010), no. 1, 201--226.

\bibitem {Ruzzi2011} Alessandro Ruzzi, {\it Smooth projective symmetric varieties with Picard number one}, Int. J. Math. {\bf 22} (2011), no. 2, 145--177.

\bibitem {Vust} Thierry Vust, {\it Op\'eration de groupes r\'eductifs dans un type de c\^ones presque homog\`enes}, Bull. Soc. Math. France {\bf 102} (1974), 317--334.


\bibitem {W} Jerzy Weyman, {\it Cohomology of vector bundles and syzygies}, Cambridge Tracks in Mathematics Vol. 149, Cambridge University Press, 2003.

\bibitem {Zak} Fyodor L. Zak,  {\it Tangents and secants of algebraic varieties}, Translations of Mathematical Monographs, Vol. 127, Amer. Math. Soc., Providence, Rhode Island, 1993.
 
\end{thebibliography}
\end{document}